\documentclass[12pt]{amsart}
\usepackage{amssymb}
\usepackage{latexsym}
\usepackage{amsmath}
\usepackage{amsthm}
\usepackage{amsfonts}
\usepackage{amscd}
\usepackage{ulem}
\usepackage{mathrsfs}
\usepackage{color}

\usepackage{amsmath,amsfonts,amssymb,amscd,amsthm}
\usepackage{mathrsfs}
\usepackage{ulem}

\usepackage{amscd}

\theoremstyle{plain}
\newtheorem{theorem}{Theorem}
\theoremstyle{plain}
\newtheorem{lemma}{Lemma}
\theoremstyle{plain}
\newtheorem{corollary}{Corollary}
\theoremstyle{plain}
\newtheorem{proposition}{Proposition}
\theoremstyle{plain}
\newtheorem{definition}{Definition}
\theoremstyle{plain}
\newtheorem{remark}{Remark}

\title
[Systems of dilations and translations in symmetric spaces]
{Representing systems of dilations and translations in symmetric spaces}

\author[S.~V.~Astashkin]{Sergey V. Astashkin}
\address{Department of Mathematics,
Samara University, Moskovskoye shosse 34, 443086,
Samara, Russia}
\email{astash56@mail.ru}

\author[P.~A.~Terekhin]{Pavel A. Terekhin}
\address{Department of Mechanics and Mathematics,
Saratov State University, Astrakhanskaya Str 83, 410012,
Saratov, Russia}
\email{terekhinpa@mail.ru}

\date{\today}


\begin{document}

\begin{abstract}
Let $X$ be an arbitrary separable symmetric space on $[0,1]$. By using a combination of the frame approach and the notion of the multiplicator space $\mathscr{M}(X)$ of $X$ with respect to the tensor product, we
investigate the problem when the sequence of dyadic dilations and translations of a function $f\in X$
is a representing system in the space $X$. The main result reads that this holds whenever $\int_0^1 f(t)\,dt\ne 0$ and $f\in \mathscr{M}(X)$. Moreover, the condition $f\in\mathscr{M}(X)$ turns out to be sharp in a certain sense. In particular, we prove that a decreasing nonnegative function $f$, $f\ne 0$, from a Lorentz space $\varLambda_{\varphi}$ generates an absolutely representing system of dyadic dilations and translations in $\varLambda_{\varphi}$ if and only if $f\in\mathscr{M}(\varLambda_{\varphi})$.
\end{abstract}

\maketitle

\section{Introduction}\label{Intro}
Let $1\le p<\infty$, $f\in L_p=L_p[0,1]$. According to the  result by Filippov and Oswald proved in \cite{FO95} the obvious necessary condition
\begin{equation}\label{main1}
\int_0^1 f(t)\,dt\ne 0
\end{equation}
assures that the sequence $\{f_{k,i}\}$ of dyadic dilations and translations of a function $f\in L_p$ defined by
$$
f_{k,i}(t)=\begin{cases}
f(2^kt-i), & t\in[\frac{i}{2^k},\frac{i+1}{2^k}],\\
0, & {elsewhere},
\end{cases}
\qquad i=0,\dots,2^k-1, \qquad k=0,1,\dots,
$$
is a representing system in the space $L_p$. This means that for every function $x\in L_p$ there is a sequence of coefficients $\{\xi_{k,i}\}$ such that $x=\sum_{k=0}^{\infty}\sum_{i=0}^{2^k-1}\xi_{k,i}f_{k,i}$ with convergence in $L_p$. A key role in the proof of this theorem is played by the following fact, which is proved also in \cite{FO95}: under condition \eqref{main1}, there is a constant $\lambda_0\in\mathbb{R}$ satisfying the inequality
\begin{equation}\label{main2}
\|1-\lambda_0 f\|_{L_p}^p:=\int_0^1 |1-\lambda_0 f(t)|^p\,dt<1.
\end{equation}

The main goal of this paper is to extend the above result to the class of symmetric spaces. By virtue of what has been said before, it would be natural to try to prove an analogue of  condition \eqref{main2} for separable symmetric spaces and then to reach desired result by using the reasoning and techniques of \cite{FO95}. Note, however, that an analogue of \eqref{main2}, which is valid in smooth spaces (see Proposition~\ref{final prop}), is far from being fulfilled in any separable symmetric space (see  Corollary~\ref{l2}, showing that \eqref{main2} does not hold in all Lorentz spaces different from $L_1$).
It made us to find another way basing on using a combination of the frame approach proposed and developed in \cite{Ter04} and the notion of the multiplicator space $\mathscr{M}(X)$ of a symmetric space $X$ with respect to the tensor product introduced and studied in \cite{Ast97}, \cite{AMS03} (for all definitions, see the next section). It is worth to emphasize that an intimate connection (though implicit, because of the obvious equation $\mathscr{M}(L_p)=L_p$, $1\le p<\infty$)
between the problem of representation of functions in symmetric spaces by dilations and translations and the notion of the multiplicator space with respect to the tensor product appears already in the paper \cite{FO95} (see Remark~\ref{Rem2}).

The main result of this paper, Theorem \ref{t2}, shows that for an arbitrary separable symmetric space $X$ the sequence of dyadic dilations and translations of every function $f\in \mathscr{M}(X)$  satisfying condition \eqref{main1} is an absolutely representing system in the space $X$. Clearly,  the Filippov-Oswald theorem, mentioned at the beginning of the Introduction, is an immediate consequence of the last assertion.

Moreover, the condition $f\in\mathscr{M}(X)$ turns out to be sharp in a certain sense. In particular, we prove that a
decreasing nonnegative function $f$, $f\ne 0$, from a Lorentz space $\varLambda_{\varphi}$ generates an absolutely representing system of dyadic dilations and translations in $\varLambda_{\varphi}$ if and only if $f\in\mathscr{M}(\varLambda_{\varphi})$ (see Theorem \ref{c1}). From this it follows the equivalence of the following conditions: (i) each function $f\in\varLambda_{\varphi}$,
$\int_0^1f(t)\,dt\neq0$, generates an absolutely representing system of dilations and translations in the Lorentz space $\varLambda_{\varphi}$ and (ii) the function $\varphi(t)$ is submultiplicative (see Corollary \ref{c2}).

In Theorem~\ref{l2}, it is shown that every frame in a symmetric space with respect to a $\ell^1$-sum of finite-dimensional spaces is not projective. In particular, this result is applicable to systems of dilations and translations.

In conclusion, the Appendix contains a discussion related to  condition \eqref{main2} and as well some remarks concerning to comparing the Weak Greedy Algorithm and the frame approach, which are used in \cite{FO95} and in the present paper, respectively (cf. \cite{Sil08}).

\section{Preliminaries}\label{Prel}

In this section, we shall briefly list the
definitions and notions used throughout this paper.

\textbf{2a. Symmetric spaces.}
For more detailed information related to symmetric spaces, we refer to the monographs \cite{BS,KPS,LT}.

A Banach space $(X,\Vert \cdot\Vert _{_{X}})$ of real-valued
Lebesgue measurable functions (with identification $m $-a.e., where $m $ is the usual Lebesgue measure) on the interval $[0,1]$ is called {\it symmetric} (or {\it rearrangement invariant}) if
\begin {enumerate}
\item[(i).]  $X$ is an ideal lattice, that is, if $y\in X$ and $x$ is any measurable function on $[0,1]$ with $\vert x\vert
\leq \vert y\vert $, then $x\in X$ and $\Vert x\Vert _{_{X}} \leq \Vert y\Vert _{_{X}};$
\item[(ii).]  $X$ is symmetric in the sense that if
 functions $x$ and $y$ are {\it equimeasurable}, i.e.,
$$
m\{u\in [0,1]: |x(u)|>s\}=m\{u\in [0,1]: |y(u)|>s\},\;\;s>0,
$$
and $y\in X$, then $x\in X$ and $\Vert x\Vert _{_{X}} = \Vert y\Vert _{_{X}}$.
\end{enumerate}
\bigskip
\noindent In particular, each measurable function $x(u)$ on $[0,1]$ is equimeasurable with its decreasing, right-continuous rearrangement $x^*(t)$ given by
$$
x^{*}(t):=\inf \{~s\ge 0:\,m\{u\in [0,1]: |x(u)|>s\}\le t~\},\quad t>0.
$$

Without loss of generality, we shall assume that any symmetric space  $X$ satisfies the condition $\|\chi_{[0,1]}\|_X=1$, where in what follows $\chi_A$ is the characteristic function of a set $A$. Then, we have $L_\infty [0,1]\subseteq X\subseteq L_1[0,1]$, $\|x\|_{L_1}\le \|x\|_{X}$, $x\in X$, and $\|x\|_{X}\le \|x\|_{L_\infty}$, $x\in L_\infty$  \cite[Theorem~II.4.1]{KPS}.

A function $\psi(t)$, $0\le t\le 1$, is called {\it quasi-concave} if $\psi(t)$ increases, $\psi(t)/t$ decreases, and $\psi(0)=0$.
For every symmetric space $X$ its {\it fundamental function} $\phi_X$, defined by  $\phi_X(t):=\|\chi_{[0,t]}\|_X$,  is quasi-concave \cite[Theorem~II.4.7]{KPS}. 

The {\it K\"othe dual} (or the {\it associated} space) $X'$ of a symmetric space $X$  consists of all measurable functions $y$, for which
$$
\Vert y\Vert _{_{X'}}:=\sup_{\|x\|_X\le1}\langle x,y\rangle<\infty,\;\;\mbox{where}\;\langle x,y\rangle:= \int _{0}^1 x(t)y(t)\,dt.
$$
If $X^*$ denotes the Banach dual of a symmetric space $X$, then $X' \subseteq X^{*}$ and $X'=X^{*}$ if and only if $X$
 is separable. A symmetric space $X$ is said to have the {\it Fatou property} if for every sequence $\{x_n\}_{n=1}^\infty\subset X$ from $x_n(t)\to x(t)$ a.e. on $[0,1]$ and $\sup _n\Vert x_n\Vert _{_{X}} <\infty $ it follows  that $x\in X$ and $\Vert x\Vert_{_{X}}\leq \liminf _{n\to \infty }\Vert x_n\Vert _{_{X}}$. It is well known that a symmetric space $X$ has the Fatou property if and only if the natural embedding of $X$ into its K\"othe bidual $X^{''}$ is a surjective isometry. We have also $\phi_{X'}(t)=t/\phi_X(t)$, $0<t\le1$ (cf. \cite[Chapter~II, (4.39)]{KPS}).

Recall that for $\tau>0$, the
{\it dilation} operator $\sigma_\tau$ is
defined by setting $\sigma_\tau
x(t)=x(t/\tau)\chi_{(0,\min(1,\tau))}(t)$, $0\le t\le 1$. Operators $\sigma_\tau$ are bounded in every symmetric space $X$ and $\|\sigma_\tau\|_{X\to X}\le\max(1,\tau)$, $\tau>0$. The numbers $\alpha_X$ and $\beta_X$ given by
\[\alpha_X:=\lim\limits_{\tau\to
0}\frac{\ln\|\sigma_\tau\|_X}{\ln\tau},\quad
\beta_X:=\lim\limits_{\tau\to
\infty}\frac{\ln\|\sigma_\tau\|_X}{\ln\tau}\]
are called the {\it Boyd indices} of $X$. Always $0\le\alpha_X\leq\beta_X\le 1$ \cite[Chapter~II, \S\,4.3]{KPS}.

Along with the classical $L_p$-spaces important examples of symmetric spaces are Lorentz, Marcinkiewicz and Orlicz spaces.
Let $\varphi(t)$ be an increasing concave function on $[0,1]$, $\varphi(0)=0$. The \textit{Lorentz} space $\varLambda_{\varphi}$
(resp. \textit{ Marcinkiewicz} space $M_{\varphi}$) consists of all measurable functions $x(t)$ on $[0,1]$
such that
$$
\|x\|_{\varLambda_{\varphi}}:=\int_0^1x^*(t)\,d\varphi(t)<\infty\;\;(\mbox{resp.}\; \|x\|_{M_{\varphi}}:=\sup_{0<t\le1}\frac{1}{\varphi(t)}\int_0^tx^*(s)\,ds<\infty).
$$
The space $\varLambda_{\varphi}$ is separable and the space $M_{\varphi}$ is not separable provided that $\lim_{t\to 0}\varphi(t)=0$ (equivalently, $\varLambda_{\varphi}\ne L_\infty$ and $M_{\varphi}\ne L_1$). At the same time, the subspace $M^0_{\varphi}$ of $M_{\varphi}$, consisting of all $x(t)$ such that
$$
\lim_{s\to 0}\frac{1}{\varphi(s)}\int_0^s x^*(t)\,dt=0,$$
is a separable symmetric space. Moreover, $(\varLambda_{\varphi})^*=(\varLambda_{\varphi})'=M_{\varphi}$  and $(M_{\varphi})' = (M^0_{\varphi})^*=\varLambda_{\varphi}$ \cite[Theorems~II.5.2 and II.5.4]{KPS}.

If $\varPhi(t)$ is an increasing convex function on $[0,\infty)$ with $\varPhi(0)=0$, then the \textit{Orlicz}
space $L_{\varPhi}$ is the set of all measurable functions $x(t)$ on $[0,1]$, for which the following Luxemburg norm
$$
\|x\|_{L_{\varPhi}}:=\inf\biggl\{\lambda>0:\int_0^1\varPhi\biggl(\frac{|x(t)|}{\lambda}\biggr)\,dt\le1\biggr\}
$$
is finite.

The notation   $A\asymp B$ will mean that there exist constants $C>0$ and $c>0$ independent of the arguments of $A$ and $B$ such that $c{\cdot}A\le B\le C{\cdot}A$.
Moreover, throughout the paper $\|f\|_p:=\|f\|_{L_p[0,1]}$, $1\le p\le\infty$.

\textbf{2b.  Multiplicator space with respect to the tensor product.}

The boundedness and other properties of the tensor product in symmetric spaces have been studied in papers \cite{O'N,Mil1,Mil2,Mil4,Ast96,Ast97,AMS03}. The notion of the multiplicator space with respect to the tensor product was introduced in \cite{Ast96}.

Let $X = X(I)$ be a symmetric space on $I=[0, 1]$. Then the
corresponding symmetric space $X(I\times I)$ on the square $I\times I$ consists of all measurable functions $x(s, t)$ on $I\times I$ such that $x^{\circledast}(t) \in X(I)$ with the norm $\|x\|_{X(I\times I)}:= \|x^{\circledast}\|_{X(I)}$, where $x^{\circledast}$ denotes the decreasing rearrangement of $|x(s,t)|$ with respect to the Lebesgue measure $m_{2}$ on $I\times I$. For two measurable functions $x(s)$ and $y(t)$ on $I$ we define the bilinear tensor product operator by $(x \otimes y)(s,t):= x(s)y(t)$, $s,t\in I$.

The {\it multiplicator space ${\mathscr M}(X)$} of a symmetric space $X$ on $I$ with respect to the tensor product is the set of all measurable functions $x(s)$
such that the operator $B_xy:=x\otimes y$ is bounded from $X$ into $X(I \times I)$.
${\mathscr M}(X)$ is a symmetric space on $I$, when it is equipped with the natural norm
\begin{equation*}
\| x \|_{{\mathscr M}(X)} := \|B_x\|_{X\to X(I \times I)}.
\end{equation*}

It is clear that $\mathscr{M}(X)=X$ if and only if the tensor product operator
$(x,y)\mapsto x\otimes y$ is a bounded mapping from $X\times X$ into $X(I\times I)$.
Since $\|x\otimes y\|_p=\|x\|_p\|y\|_p$ (by Fubini theorem), then $\mathscr{M}(L_p)=L_p$ for all $1\le p\le\infty$.

Here, we list known results, identifying multiplicator spaces for some classes of symmetric spaces (see \cite{Ast96,Ast97,AMS03}):

(i) $\mathscr{M}(\varLambda_{\varphi})=\varLambda_{\varphi}$ if and only if the function $\varphi$ is submultiplicative,
i.e., $\varphi(st)\le C\varphi(s)\varphi(t)$, for some $C>0$ and all $0\le s,t\le1$;

(ii) $\mathscr{M}(M_{\varphi})=M_{\varphi}$ if and only if $\varphi'\otimes\varphi'\in M_{\varphi}(I\times I)$; in particular, the latter condition holds if $\varphi(t)\le C\varphi(t^2)$, $0\le t\le1$;

(iii) $\mathscr{M}(L_{\varPhi})=L_{\varPhi}$ if and only if the function $\varPhi$ is submultiplicative for $t$ large enough, i.e., there is $t_0>0$ such that $\varPhi(st)\le C\varPhi(s)\varPhi(t)$ for all $s,t\ge t_0$.

For every symmetric space $X$ we have $\varLambda_{\varphi}\subset\mathscr{M}(X)\subset L_p$, where $\varphi(t)=\|\sigma_t\|_{X\to X}$, $0<t\le1$,
and $p=1/\alpha_X$, with embedding constants independent of $X$. In particular, $\mathscr{M}(X)=L_{\infty}$
if and only if $\alpha_X=0$ (see \cite{Ast97}).

Observe that $\mathscr{M}(\mathscr{M}(X))=\mathscr{M}(X)$, where symmetric space $X$ is arbitrary \cite[Proposition~1]{AMS03}. Therefore, $\mathscr{M}(X)=X$ whenever
$X=\mathscr{M}(Y)$ for some symmetric space $Y$, However, in general, the embedding $X\subset Y$ does not imply that $\mathscr{M}(X)\subset\mathscr{M}(Y)$ \cite[p.~252]{AMS03}.


\textbf{2c. Representing systems and frames in Banach spaces.}
A sequence $\{x_n\}_{n=1}^{\infty}$ of elements of a Banach space $X$ is said to be a \textit{representing system}
if for each $x\in X$ we can find a sequence of coefficients $\{\xi_n\}_{n=1}^{\infty}$ such that $x=\sum_{n=1}^{\infty}\xi_nx_n$.

There is a close connection between the latter notion and the following definition of the frame in Banach spaces that was introduced and developed in \cite{Ter04}.

Let $\varDelta$ be a Banach space of sequences $\xi=\{\xi_n\}_{n=1}^{\infty}$ such that the standard unit vectors
$\delta_n$, $n=1,2,\dots$, form a basis
in $\varDelta$.
Then the dual space $\varDelta^*$, clearly, can be identified with the Banach space of sequences $\eta=\{\eta_n\}_{n=1}^{\infty}$ such that
$$
\|\eta\|_{\varDelta^*}:=\sup_{\|\xi\|_{\varDelta}\le1}|\langle\xi,\eta\rangle|<\infty,\;\;\mbox{where}\;
\langle\xi,\eta\rangle:=\sum_{n=1}^{\infty}\xi_n\eta_n.
$$

\begin{definition}\label{d1}
We say that a sequence $\{x_n\}_{n=1}^{\infty}$ of nonzero elements of a Banach space $X$ is a frame in $X$
with respect to $\varDelta$ whenever there exist constants $0<A\le B<\infty$ such that for all $y\in X^*$ the following inequalities hold
\begin{equation}\label{eq2.1}
A\|y\|_{X^*}\le\|\{\langle x_n,y\rangle\}_{n=1}^{\infty}\|_{\varDelta^*}\le B\|y\|_{X^*}.
\end{equation}
\end{definition}
In particular, if $X$ is a Hilbert space and $\varDelta=\ell^2$, we get the definition of the Duffin-Shaeffer frame.
In the general case of Banach spaces this definition is dual with respect to the well-known definitions of the atomic decomposition and the frame due to Gr\"{o}chenig \cite{Gr91} and also to some other close notions (cf. \cite{CHL99}). We emphasize that whenever we talk  in this paper about the frame, it will be understood in the sense of Definition~\ref{d1}.

The mapping $S:\varDelta\to X$ defined by
$$
S\xi:=\sum_{n=1}^{\infty}\xi_nx_n
$$
is called the \textit{synthesis operator}. Respectively, the mapping $R:X^*\to \varDelta^*$,
$$
Ry:=\{\langle x_n,y\rangle\}_{n=1}^{\infty},
$$
is the \textit{analysis operator}. One can easily check that $S^*=R$.

The following result is proved in \cite{Ter10} (see also \cite{Ter04}).

\begin{proposition}\label{p1}
A sequence $\{x_n\}_{n=1}^{\infty}\subset X\setminus\{0\}$ is a frame in a Banach space $X$ with respect to $\varDelta$ if and only if the following conditions are satisfied:

$(i)$ for every $\xi=\{\xi_n\}_{n=1}^{\infty}\in\varDelta$ the series $\sum_{n=1}^{\infty}\xi_nx_n$ converges in $X$;

$(ii)$ for every $x\in X$ there is $\xi=\{\xi_n\}_{n=1}^{\infty}\in\varDelta$ such that $x=\sum_{n=1}^{\infty}\xi_nx_n$.
\end{proposition}

From Proposition \ref{p1} it follows that every frame in a Banach space $X$ (understood as in Definition \ref{d1}) is both a representing system in $X$. The converse holds
as well; each representing system in a Banach space $X$ is also a frame in $X$ with respect to some sequence space, which is defined, in general, in contrast to a basis, in a non-unique way (cf. \cite{Ter04,Ter10}).


\textbf{2d. Operator approach to studying systems of dilations and translations.}

Further, we shall make use of the following operator approach to the definition of systems of dilations and translations.

For a function $f\in L_1[0,1]$ we set
$$
V_0f(t)=\begin{cases}
f(2t), & 0\le t\le\frac12,\\
0, & \frac12<t\le1,
\end{cases}
\qquad
V_1f(t)=\begin{cases}
0, & 0\le t<\frac12,\\
f(2t-1), & \frac12\le t\le1.
\end{cases}
$$
Observe that $V_0$ coincides with the dilation operator $\sigma_{1/2}$ and $V_1x$ is a translation of the function
$\sigma_{1/2}x$ for each $x\in L_1[0,1]$. Hence, the operators $V_0$ and $V_1$ are bounded on every symmetric space $X$ and moreover $\|V_i\|_{X\to X}\le1$, $i=0,1$. Also, we define the operator $W$ by
$$
Wf(t)=(V_0+V_1)f(t)=f(2t \mod [0,1]).
$$

Denote
$$
\mathbb{A}=\bigcup_{k=0}^{\infty}\{0,1\}^k,
$$
that is, the family $\mathbb{A}$ consists of all multi-indices $\alpha=(\alpha_1,\dots,\alpha_k)$ such that $\alpha_\nu=0$ or
$1$, $\nu=1,\dots,k$, $k=0,1,2,\dots$ Also, in what follows $|\alpha|(=k)$ is the length of a multi-index $\alpha=(\alpha_1,\dots,\alpha_k)\in\mathbb{A}$,
$\alpha\beta$ is the concatenation of $\alpha=(\alpha_1,\dots,\alpha_k)$ and $\beta=(\beta_1,\dots,\beta_l)$, i.e., the multi-index $(\alpha_1,\dots,\alpha_k,\beta_1,\dots,\beta_l)$,
and $\{\xi_\alpha\}_{\alpha\in\mathbb{A}}$ (or $\xi(\alpha)$, $\alpha\in \mathbb{A}$) is a real-valued function defined on $\mathbb{A}$.

Now, setting for any $\alpha=(\alpha_1,\dots,\alpha_k)\in\mathbb{A}$
$$
V^{\alpha}f:=V_{\alpha_1}\dots V_{\alpha_k}f,
$$
we get the system of dyadic dilations and translations of a function $f$, which will be denoted further by $\{V^{\alpha}f\}_{\alpha\in\mathbb{A}}$. Clearly, in the usual notation, $V^{\alpha}f:=f_{k,i}$, where $k=|\alpha|$ and $i=\sum_{\nu=1}^k\alpha_{\nu}2^{k-\nu}$.
In particular, $\{V^{\alpha}1\}_{\alpha\in\mathbb{A}}$ (i.e., when $f(t)\equiv 1$)
is just the sequence of characteristic functions of dyadic intervals $I_{\alpha}=[{i}{2^{-k}},(i+1){2^{-k}}]$,
$i=0,1,\dots,2^k-1$, $k=0,1,2,\dots$ In turn, for the function
$$
h(t)=\begin{cases}
1, & 0<t<\frac12,\\
-1, & \frac12<t<1,
\end{cases}
$$
the system $\{V^{\alpha}h\}_{\alpha\in\mathbb{A}}$ coinsides with the classical Haar system normed in $L_{\infty}$ (without the first function equal to
$1$).

It turns out that certain conditions allow to compare norms of linear combinations of the functions $V^{\alpha}f$, $|\alpha|=k$, and $V^{\alpha}1$, $|\alpha|=k$, with a fixed $k\in \mathbb{N}$. Specifically, the condition
$f\in\mathscr{M}(X)$ is sufficient (and necessary if $X$ is separable) for the inequality
\begin{equation}\label{eq2.2}
\biggl\|\sum_{|\alpha|=k}\xi_{\alpha}V^{\alpha}f\biggr\|_X\le C\biggl\|\sum_{|\alpha|=k}\xi_{\alpha}V^{\alpha}1\biggr\|_X.
\end{equation}
to hold, for some constant
$C>0$ and all $\xi_{\alpha}\in\mathbb{R}$, $|\alpha|=k$, $k=0,1,\dots$ Moreover, we can take in \eqref{eq2.2} $C=\|f\|_{\mathscr{M}(X)}$. 

Indeed, the function $U_f:=\sum_{|\alpha|=k}\xi_{\alpha}V^{\alpha}f$
is equimeasurable with the tensor product of $f$ and the step function $U=\sum_{|\alpha|=k}\xi_{\alpha}V^{\alpha}1$,
because for each $\tau>0$
\begin{align*}
m\{t\in[0,1]:|U_f(t)|>\tau\}=\sum_{|\alpha|=k}m\{t\in I_{\alpha}:|\xi_{\alpha}W^kf(t)|>\tau\} \\
=\frac{1}{2^k}\sum_{|\alpha|=k}m\{t\in[0,1]:|\xi_{\alpha}f(t)|>\tau\}
=m_2\{s,t\in[0,1]:|f(s)U(t)|>\tau\}.
\end{align*}
Therefore, $\|U_f\|_X=\|f\otimes U\|_{X(I\times I)}\le \|f\|_{\mathscr{M}(X)}\|U\|_X$ and we get \eqref{eq2.2}.

Conversely, if inequality \eqref{eq2.2} holds, or equivalently $\|f\otimes U\|_{X(I\times I)}\le C\|U\|_X$ for each step function $U$, then, assuming that $X$ is separable, we easily have $\|f\otimes x\|_{X(I\times I)}\le C\|x\|_X$ for all $x\in X$, whence $\|f\|_{\mathscr{M}(X)}\le C$.

On the other hand, one can easily see (cf. \cite[Theorem 6]{AMS03}) that the opposite inequality to \eqref{eq2.2} holds for every function $f\in X$, $f\neq0$. More precisely, there is a constant $c_f>0$ such that for all $k=0,1,\dots$ and
$\xi_{\alpha}\in\mathbb{R}$, $|\alpha|=k$, we have
\begin{equation}\label{eq2.3}
c_f\biggl\|\sum_{|\alpha|=k}\xi_{\alpha}V^{\alpha}1\biggr\|_X\le\biggl\|\sum_{|\alpha|=k}\xi_{\alpha}V^{\alpha}f\biggr\|_X.
\end{equation}

Further, we shall repeatedly use estimates \eqref{eq2.2} and \eqref{eq2.3}.


\section{Frames in symmetric spaces}\label{Frame}

Let $X$ be a symmetric space on $[0,1]$, $f\in X$. Denote by $X_{k,f}$ the linear span in $X$ of the set of dilations and translations of $f$ supported on the dyadic intervals of rank $k$, i.e., $X_{k,f}:=\text{span}(\{V^{\alpha}f\}_{|\alpha|=k})$. Then, the normalization condition
\begin{equation}\label{eq3.1}
\langle f,1\rangle=\int_0^1f(t)\,dt=1
\end{equation}
assures that, for each $k=0,1,\dots$, the operator $P_{k,f}$, defined by
$$
P_{k,f}x:=2^k\sum_{|\alpha|=k}\langle x,V^{\alpha}1\rangle V^{\alpha}f,
$$
is a projection from the space $X$ onto $X_{k,f}$.
In particular, if $f\equiv 1$, we get the subspace of dyadic step functions of rank $k$ and the classical average projection that will be denoted by $X_k$ and $P_k$, respectively. In this special case we have $X_k\subset X_{k+1}$ and $\|P_k\|_{X\to X}=1$ (see e.g. \cite[Ch.~II, \S\,3-4]{KPS}). Moreover, if a symmetric space $X$ is separable, the sequence $\{P_kx\}_{k=1}^\infty$ converges in norm to $x$ for each $x\in X$ \cite[Theorem~II.4.3]{KPS}.

In the case when $f$ is an arbitrary function, the sequence of  norms $\|P_{k,f}\|_{X\to X}$ does not decrease and, in general, may be unbounded. It is worth to mention that for every separable symmetric space $X$ the condition
\begin{equation}\label{eq3.2}
\sup_{k\ge0}\|P_{k,f}\|_{X\to X}<\infty
\end{equation}
is equivalent to the fact that $f$ belongs to the multiplicator $\mathscr{M}(X)$ \cite[Corollary~2]{AMS03}. In the case when $X$ is an arbitrary symmetric space, we have only the implication: from $f \in\mathscr{M}(X)$ it follows \eqref{eq3.2} \cite[Theorem~2(i)]{AMS03}.

It turns out that the projections $P_{k,f}$ like $P_k$ possess approximate properties but now with respect to the weak topology of a separable symmetric space.

\begin{lemma}\label{l1}
Let $X$ be a separable symmetric space and let a function $f\in X$ satisfy condition \eqref{eq3.1}.
The following conditions are equivalent:

$(i)$ the functions $P_{k,f}x$ converge weakly to $x$ for each $x\in X$;

$(ii)$ $f\in\mathscr{M}(X)$.
\end{lemma}

\begin{proof}
Implication $(i)\Rightarrow(ii)$ is almost immediate. Indeed, from $(i)$ it follows that for any $x\in X$
$$
\sup_{k\ge0}\|P_{k,f}x\|_X<\infty.
$$
Therefore, by Uniform Boundedness Principle, we have \eqref{eq3.2}. As was observed above, this implies that $f\in\mathscr{M}(X)$.

$(ii)\Rightarrow(i)$. Suppose $f\in\mathscr{M}(X)$. We claim that the functions $P_{k,f}1$ converge weakly to $1$ as $k\to\infty$.

Firstly, observe that
$$
P_{k,f}1=\sum_{|\alpha|=k}V^{\alpha}f=W^kf,
$$
where as above
$$
Wf(t)=(V_0+V_1)f(t)=f(2t \mod [0,1]).
$$
Further, let us recall the following classical Fejer lemma (see e.g. \cite[\S~20]{Bar}): for any $1$-periodic functions $y\in L^1$ and $z\in L^{\infty}$ it holds
$$
\lim_{n\to\infty}\int_0^1y(t)z(nt)\,dt=\int_0^1y(t)\,dt\int_0^1z(t)\,dt.
$$
Since $X$ is separable, then $X^*=X'\subset L^1$. Therefore, applying this relation to every $y\in X^*$ and $z=P_if\in L^{\infty}$, with a fixed $i=1,2,\dots$, we have
$$
\lim_{k\to\infty}\int_0^1y(t)W^k(P_if)(t)\,dt=\int_0^1y(t)\,dt\int_0^1P_if(t)\,dt,\;\;i=1,2,\dots
$$
From normalization condition \eqref{eq3.1} it follows
$$
\int_0^1P_if(t)\,dt=\int_0^1f(t)\,dt=1,\;\;i=1,2,\dots,$$
and hence the left-hand side of the preceding equation does not depend on $i$. On the other hand, since $\|P_if-f\|_X\to 0$ as $i\to\infty$, we easily get
$$
\lim_{i\to\infty}\int_0^1y(t)W^k(P_if)(t)\,dt=\int_0^1y(t)W^kf(t)\,dt,\;\;k=0,1,2,\dots$$
Combining these relations, we conclude that
$$
\lim_{k\to\infty}\langle W^kf,y\rangle=\langle1,y\rangle
$$
for every $y\in X^*$. Equivalently, $P_{k,f}1\to 1$ weakly as $k\to\infty$, and so our claim is proved.

Further, let $k,l\in\mathbb{N}$, $l\le k$, and let $\beta$ be a multi-index such that $|\beta|=l$. Since $\langle V^{\beta}x,V^{\alpha}1\rangle=2^{-l}\langle x,V^{\gamma}1\rangle$ if $\alpha=\beta\gamma$, $|\gamma|=k-l$, and $\langle V^{\beta}x,V^{\alpha}1\rangle=0$, otherwise, we have
$$
P_{k,f}V^{\beta}x=2^k\sum_{|\alpha|=k}\langle V^{\beta}x,V^{\alpha}1\rangle V^{\alpha}f
=2^{k-l}\sum_{|\gamma|=k-l}\langle x,V^{\gamma}1\rangle V^{\beta\gamma}f=V^{\beta}P_{k-l,f}x.
$$
Therefore, for every dyadic step function
$x=\sum_{|\beta|=l}\xi_{\beta}V^{\beta}1$, $l=0,1,\dots$ and $k\in\mathbb{N}$ such that $l\le k$ it holds
$$
P_{k,f}x=\sum_{|\beta|=l}\xi_{\beta}V^{\beta}P_{k-l,f}1.
$$
From the fact that $P_{k-l,f}1$  converge weakly to  $1$ as $k\to\infty$, we deduce weak convergence of the functions $V^{\beta}P_{k-l,f}1$  to  $V^{\beta}1$ for each multi-index $\beta$. Thus, $P_{k,f}x\to x$ weakly for each dyadic step function $x\in X$. Moreover, from the hypothesis $f\in\mathscr{M}(X)$ it follows that condition \ref{eq3.2} holds. As a result, since $X$ is separable, the set of dyadic step functions is dense in $X$ and hence $P_{k,f}x\to x$ weakly for every $x\in X$.
\end{proof}

In what follows, $\varXi_k$, $k=0,1,\dots$, are coordinate spaces of dimension $2^k$, whose elements are sequences $\xi=\{\xi_{\alpha}\}_{|\alpha|=k}$ of reals. When  $\varXi_k$ is equipped with the norm
$$
\|\xi\|_{\varXi_k}:=\biggl\|\sum_{|\alpha|=k}\xi_{\alpha}V^{\alpha}1\biggr\|_X,
$$
it is clearly isometric to the subspace $X_k$ of dyadic step functions of rank $k$ in $X$.

Our first result shows a direct link between systems of dilations and translations in a separable symmetric space $X$ and frames in $X$ (see Definition~\ref{d1}) with respect to a suitable sequence space.

\begin{theorem}\label{t1}
Let $X$ be a separable symmetric space and let a function  $f\in X$ satisfy condition \eqref{main1}. Suppose also that $f\in\mathscr{M}(X)$.

Then, the system of dilations and translations of $f$ is a frame in $X$ with respect to the Banach sequence space
$$
\varXi=\biggl(\bigoplus_{k=0}^{\infty}\varXi_k\biggr)_{\ell^1}.
$$
\end{theorem}

\begin{proof}
We need to show that for some $0<A\le B<\infty$ and all $y\in X^*$
\begin{equation}\label{eq3.4a}
A\|y\|_{X^*}\le\|\{\langle V^{\alpha}f,y\rangle\}_{\alpha\in\mathbb{A}}\|_{\varXi^*}\le B\|y\|_{X^*}.
\end{equation}
Firstly, since $\varXi$ is the $\ell^1$-sum of the spaces $\varXi_k$, the dual space $\varXi^*$ is the $\ell^{\infty}$-sum of the dual spaces $(\varXi_k)^*$, i.e.,
$$
\varXi^*=\biggl(\bigoplus_{k=0}^{\infty}(\varXi_k)^*\biggr)_{\ell^{\infty}}.
$$
Secondly, if $(\varXi^*)_k$ is the coordinate space, corresponding to the symmetric space $X^*$ (recall that $X$ is separable and so $X^*=X'$), then
\begin{equation}\label{sequence}
\|\eta\|_{(\varXi_k)^*}=\|2^k\eta\|_{(\varXi^*)_k}, \qquad k=0,1,\dots.
\end{equation}

Indeed, one can easily check that for every $u\in X$, $v\in X^*$ and $k\in\mathbb{N}$ 
$$
\langle u,P_kv\rangle = \langle P_ku,v\rangle.
$$
Therefore, since $\|P_k\|_{X\to X}=1$ and for any step function $y=\sum_{|\alpha|=k}\eta_{\alpha}V^{\alpha}1$ it holds $y=P_ky$, we get
$$
\|y\|_{X^*}=\sup_{x\in X,\,\|x\|_{X}\le1}|\langle x,y\rangle|=\sup_{x\in X_k,\,\|x\|_{X}\le1}|\langle x,y\rangle|.
$$
Here, as above, $X_k$ is the subspace of  dyadic step functions supported on dyadic intervals of rank $k$.
Combining this together with the fact that
$\langle x,y\rangle=2^{-k}\langle\xi,\eta\rangle$ for any function $x=\sum_{|\alpha|=k}\xi_{\alpha}V^{\alpha}1\in X_k$, we have
$$
\|\eta\|_{(\varXi_k)^*}=\sup_{\|\xi\|_{\varXi_k}\le1}|\langle\xi,\eta\rangle|
=2^k\sup_{x\in X_k,\,\|x\|_X\le1}|\langle x,y\rangle|=2^k\|y\|_{X^*}=\|2^k\eta\|_{(\varXi^*)_k}.
$$
From these observations it follows that inequalities \eqref{eq3.4a} can be rewritten as follows:
\begin{equation}\label{eq3.4}
A\|y\|_{X^*}\le\sup_{k\ge0}\biggl\|2^k\sum_{|\alpha|=k}\langle V^{\alpha}f,y\rangle V^{\alpha}1\biggr\|_{X^*}\le B\|y\|_{X^*}.
\end{equation}

Further, using Lemma \ref{l1} and taking into account normalization condition \eqref{eq3.1}, we infer that the sequence $\{P_{k,f}x\}$ converges weakly to $\langle f,1\rangle x$ for all $x\in X$. Therefore, the sequence $\{P_{k,f}^*y\}$ converges weakly* to $\langle f,1\rangle y$ for all $y\in X^*$ and hence
\begin{equation}\label{eq3.5}
|\langle f,1\rangle|\|y\|_{X^*}\le\liminf_{k\to\infty}\|P_{k,f}^*y\|_{X^*}.
\end{equation}
Moreover, one can easily check that
$$
P_{k,f}^*y=2^k\sum_{|\alpha|=k}\langle V^{\alpha}f,y\rangle V^{\alpha}1.
$$
Consequently, estimate \eqref{eq3.5} implies the left-hand side inequality in \eqref{eq3.4} with the constant $A=|\langle f,1\rangle|>0$.
In turn, the right-hand side inequality is a consequence of the hypothesis $f\in\mathscr{M}(X)$, because
$$\|P_{k,f}^*\|_{X^*\to X^*}=\|P_{k,f}\|_{X\to X},\;\;k=1,2,\dots,$$ and so the norms $\|P_{k,f}^*\|_{X^*\to X^*}$ are uniformly bounded. Finally, taking into account estimate \eqref{eq2.2}, we can take $B=\|f\|_{\mathscr{M}(X)}<\infty$.

\end{proof}

\section{Representation of functions in symmetric spaces by dilations and translations.}

\begin{definition}\label{d2}
Let $X$ be a symmetric space on $[0,1]$. We say that the sequence of dilations and translations of a function $f\in X$ is an absolutely representing system in $X$ with the constant $C>0$ if for every $x\in X$ there exist coefficients $\{\xi_{\alpha}\}_{\alpha\in\mathbb{A}}$ such that we have
$$
x=\sum_{\alpha\in\mathbb{A}}\xi_{\alpha}V^{\alpha}f\;\;\mbox{and}\;\;
\sum_{k=0}^{\infty}\biggl\|\sum_{|\alpha|=k}\xi_{\alpha}V^{\alpha}f\biggr\|_X\le C\|x\|_X.
$$
\end{definition}

Observe that, for each fixed $k=0,1,2,\dots$, the functions $V^{\alpha}f$, $|\alpha|=k$, are pairwise disjoint. Therefore, the inequality from Definition~\ref{d2} guarantees that any absolutely representing system of dilations and translations, generated by a function $f$ from a symmetric space $X$, is unconditional.

The following main result of the paper establishes close connections between the multiplicator space of a separable symmetric space $X$ and representation properties of systems of dilations and translations generated by functions from $X$.

\begin{theorem}\label{t2}
Let $X$ be a separable symmetric space and let a function  $f\in X$ satisfy condition \eqref{main1}.

The following conditions are equivalent:

(i) $f\in\mathscr{M}(X)$;

(ii) for any $k_0\ge0$ the sequence $\{V^{\alpha}f\}_{|\alpha|\ge k_0}$ 
is an absolutely representing system in $X$ with a constant $C$ independent of $k_0$.
\end{theorem}

\begin{proof}
$(i)\Rightarrow(ii)$. First of all, as above, the hypothesis implies inequality \eqref{eq3.5}. Furthermore, an inspection of the proof of Theorem \ref{t1} shows that, for each $k_0=0,1,2,\dots$, from \eqref{eq3.5} it follows
$$
A\|y\|_{X^*}\le\sup_{k\ge k_0}\biggl\|2^k\sum_{|\alpha|=k}\langle V^{\alpha}f,y\rangle V^{\alpha}1\biggr\|_{X^*}\le B\|y\|_{X^*},
$$
with the same constants $A=|\langle f,1\rangle|$ and $B=\|f\|_{\mathscr{M}(X)}$. This means that the analysis operator
\begin{equation}\label{analysis}
Ry:=\{\langle V^{\alpha}f,y\rangle\}_{|\alpha|\ge k_0}
\end{equation}
is an injection from $X^*$ into the Banach sequence space
$$
\biggl(\bigoplus_{k=k_0}^{\infty}(\varXi_k)^*\biggr)_{\ell^{\infty}}.
$$
Observe that $S^*=R$, where $S$ is the synthesis operator defined by
\begin{equation}\label{synthesis}
S\xi=\sum_{k=k_0}^{\infty}\sum_{|\alpha|=k}\xi_{\alpha}V^{\alpha}f.
\end{equation}
Therefore, by the well-known duality of injection and surjection properties \cite[B.3.9]{Pie}, $S$ is a surjection from the pre-dual Banach sequence space
$$
\varXi^{[k_0]}=\biggl(\bigoplus_{k=k_0}^{\infty}\varXi_k\biggr)_{\ell^1}
$$
onto $X$. Thus, for every $x\in X$ there is a sequence $\xi$ such that $x=S\xi$ and $\|\xi\|_{\varXi^{[k_0]}}\le A^{-1}\|x\|_X$. Equivalently, by estimate \eqref{eq2.2}, each function $x\in X$ admits a representation
$$
x=\sum_{k=k_0}^{\infty}\sum_{|\alpha|=k}\xi_{\alpha}V^{\alpha}f,
$$
with coefficients satisfying
$$
\sum_{k=k_0}^{\infty}\biggl\|\sum_{|\alpha|=k}\xi_{\alpha}V^{\alpha}f\biggr\|_X\le
\|f\|_{\mathscr{M}(X)}\sum_{k=k_0}^{\infty}\biggl\|\sum_{|\alpha|=k}\xi_{\alpha}V^{\alpha}1\biggr\|_X
=B\|\xi\|_{\varXi^{[k_0]}}\le BA^{-1}\|x\|_X.
$$
So, $\{V^{\alpha}f\}_{|\alpha|\ge k_0}$ is an absolutely representing system for $X$ with the constant
$C=BA^{-1}:=\|f\|_{\mathscr{M}(X)}|\langle f,1\rangle|^{-1}$.

$(ii)\Rightarrow(i)$. Let us repeat reasoning from the first part of the proof but in the opposite direction, replacing the space $\varXi^{[k_0]}$ (independent of $f$) by the space
$$
\varXi_f^{[k_0]}:=\biggl(\bigoplus_{k=k_0}^{\infty}\varXi_{k,f}\biggr)_{\ell^1},
$$
where $\varXi_{k,f}$ are coordinate copies of the subspaces $X_{k,f}$. Then condition $(ii)$ means that the synthesis operator $S$ (see \eqref{synthesis})
is an surjection from $\varXi_f^{[k_0]}$ onto $X$, and for any $x\in X$ there exists a sequence $\xi$ such that $x=S\xi$
and $\|\xi\|_{\varXi_f^{[k_0]}}\le C\|x\|_X$, where $C$ does not depend on $k_0$. Observe that $S^*=R$, where $R$ is the analysis operator \eqref{analysis}.
Consequently, $R$ is an injection from $X^*$ into $\big(\varXi_f^{[k_0]}\big)^*$ and $\|y\|_{X^*}\le C\|Ry\|_{(\varXi_f^{[k_0]})^*}$ for all $y\in X^*$ \cite[B.3.9]{Pie}. Equivalently, the last inequality can be rewritten as follows:
\begin{equation}\label{eq3.6}
\|y\|_{X^*}\le C\sup_{k\ge k_0}
\sup_{\|\{\xi_{\alpha}\}_{|\alpha|=k}\|_{\varXi_{k,f}}\le1}\biggl|\sum_{|\alpha|=k}\xi_{\alpha}\langle V^{\alpha}f,y\rangle\biggr|.
\end{equation}

Further, given sequence $\eta=\{\eta_{\beta}\}_{|\beta|=k_0}$ we set $y=\sum_{|\beta|=k_0}\eta_{\beta}V^{\beta}1$.
Since
$$
\langle V^{\alpha}f,V^{\beta}1\rangle=2^{-k_0}\langle V^{\gamma}f,1\rangle=2^{-k}\langle f,1\rangle$$
if $\alpha=\beta\gamma$, $|\gamma|=k-k_0$, and $\langle V^{\alpha}f,V^{\beta}1\rangle=0$, otherwise, we have
\begin{equation}\label{tecnique}
\sum_{|\alpha|=k}\xi_{\alpha}\langle V^{\alpha}f,y\rangle
=\sum_{|\alpha|=k}\sum_{|\beta|=k_0}\xi_{\alpha}\eta_{\beta}\langle V^{\alpha}f,V^{\beta}1\rangle
=\frac{\langle f,1\rangle}{2^k}\sum_{|\beta|=k_0}\sum_{|\gamma|=k-k_0}\xi_{\beta\gamma}\eta_{\beta}.
\end{equation}
Observe that $\varXi_{k,f}$ is a symmetric sequence space for each $k\in\mathbb{N}$. Therefore, similarly as in the proof of Theorem \ref{t1}, when calculating
$$
\sup_{\|\{\xi_{\alpha}\}_{|\alpha|=k}\|_{\varXi_{k,f}}\le1}
\biggl|\frac{1}{2^k}\sum_{|\beta|=k_0}\sum_{|\gamma|=k-k_0}\xi_{\beta\gamma}\eta_{\beta}\biggr|,
$$
we can additionally assume that
$\xi_{\beta\gamma}=\xi_{\beta}$ for all multi-indices $\beta$ and $\gamma$ satisfying $|\beta|=k_0$ and $|\gamma|=k-k_0$. Observe that for such a sequence $\{\xi_{\alpha}\}_{|\alpha|=k}$ we have
$$
\sum_{|\alpha|=k}\xi_{\alpha}V^{\alpha}f
=\sum_{|\beta|=k_0}\xi_{\beta}\sum_{|\gamma|=k-k_0}V^{\beta\gamma}f
=\sum_{|\beta|=k_0}\xi_{\beta}V^{\beta}W^{k-k_0}f.
$$
Moreover, since the functions $f(t)$ and $W^{k-k_0}f(t)=f(2^{k-k_0}t \mod [0,1])$ are equimeasurable, the functions $\sum_{|\beta|=k_0}\xi_{\beta}V^{\beta}W^{k-k_0}f$ and $\sum_{|\beta|=k_0}\xi_{\beta}V^{\beta}f$ are equimeasurable as well. Therefore, we get
$$
\|\{\xi_{\alpha}\}_{|\alpha|=k}\|_{\varXi_{k,f}}=\biggl\|\sum_{|\alpha|=k}\xi_{\alpha}V^{\alpha}f\biggr\|_X
=\biggl\|\sum_{|\beta|=k_0}\xi_{\beta}V^{\beta}f\biggr\|_X=\|\{\xi_{\beta}\}_{|\beta|=k_0}\|_{\varXi_{k_0,f}}.
$$
From the above observations it follows that 
\begin{equation*}\label{eq3.7}
\sup_{\|\{\xi_{\alpha}\}_{|\alpha|=k}\|_{\varXi_{k,f}}\le1}
\biggl|\frac{1}{2^k}\sum_{|\beta|=k_0}\sum_{|\gamma|=k-k_0}\xi_{\beta\gamma}\eta_{\beta}\biggr|
=\sup_{\|\{\xi_{\beta}\}_{|\beta|=k_0}\|_{\varXi_{k_0,f}}\le1}
\biggl|\frac{1}{2^{k_0}}\sum_{|\beta|=k_0}\xi_{\beta}\eta_{\beta}\biggr|.
\end{equation*}
Combining this together with \eqref{eq3.6} and \eqref{tecnique}, we obtain
$$
\|y\|_{X^*}\le C|\langle f,1\rangle|
\sup_{\|\{\xi_{\beta}\}_{|\beta|=k_0}\|_{\varXi_{k_0,f}}\le1}
\biggl|\frac{1}{2^{k_0}}\sum_{|\beta|=k_0}\xi_{\beta}\eta_{\beta}\biggr|
=\frac{C|\langle f,1\rangle|}{2^{k_0}}\|\eta\|_{(\varXi_{k_0,f})^*}.
$$
On the other hand, according to \eqref{sequence} (using the same notation), we have 
$$
\|y\|_{X^*}=\|\eta\|_{(\varXi^*)_{k_0}}=2^{-k_0}\|\eta\|_{(\varXi_{k_0})^*}.$$
Therefore, from the preceding inequality it follows that $(\varXi_{k_0,f})^*\subset(\varXi_{k_0})^*$ with the constant $C|\langle f,1\rangle|$, whence $\varXi_{k_0}\subset\varXi_{k_0,f}$, with the same constant. Equivalently, for every $k_0\in\mathbb{N}$ we get the estimate
$$
\biggl\|\sum_{|\beta|=k_0}\xi_{\beta}V^{\beta}f\biggr\|_X\le C|\langle f,1\rangle|
\biggl\|\sum_{|\beta|=k_0}\xi_{\beta}V^{\beta}1\biggr\|_X,
$$
and so $f\in\mathscr{M}(X)$ and
$\|f\|_{\mathscr{M}(X)}\le C|\langle f,1\rangle|$ (see the discussion related to inequality \eqref{eq2.2} in Section~\ref{Prel}d).
\end{proof}

\begin{remark}
Since $\mathscr{M}(L_p)=L_p$, $1\le p\le\infty$, the Filippov-Osvald theorem from the paper \cite{FO95} (see Section~\ref{Intro}) is a very special case of Theorem \ref{t2}. Observe also that, similarly as in the latter theorem, condition \eqref{main2} assures that the sequence of dilations and translations $\{f_{k,i}\}$ of a function $f\in L_p$, $\int_0^1f(t)\,dt\neq0$, is an  absolutely representing system in $L_p$, i.e., for each $x\in L_p$ there is a sequence of coefficients $\{\xi_{k,i}\}$ such that
$$
x=\sum_{k=0}^{\infty}\sum_{i=0}^{2^k-1}\xi_{k,i}f_{k,i}\;\;\mbox{and}\;\;\sum_{k=0}^{\infty}\Big\|\sum_{i=0}^{2^k-1}\xi_{k,i}f_{k,i}\Big\|_{L_p}<\infty.$$
\end{remark}

Using Theorem \ref{t2} and the results on multiplicator spaces for the main classes of symmetric spaces listed in Section~\ref{Prel}b (see also the discussion at the beginning of Section~\ref{Frame} and \cite[Theorem~5]{AMS03}), we obtain

\begin{corollary}
If $X$ is a symmetric space, then the sequence of dilations and translations of every function $f$, $\int_0^1f(t)\,dt\neq0$, from the Lorentz space
$\varLambda_{\varphi}$, where $\varphi(t)=\|\sigma_t\|_{X\to X}$, $0<t\le1$, is an absolutely representing system in $X$.
\end{corollary}

\begin{corollary}
(a) Let $\varphi$ be an increasing concave function on $[0,1]$, $\varphi(0)=0$, and let $\varphi'\otimes\varphi'\in M_{\varphi}(I\times I)$. Then, the sequence of dilations and translations of every function $f\in M^0_{\varphi}$, $\int_0^1f(t)\,dt\neq0$, is an absolutely representing system in the space $M^0_{\varphi}$. In particular, this holds if $\varphi(t)\le C\varphi(t^2)$, for some $C>0$ and all $0\le t\le1$.

(b) Let $\varPhi(t)$ be an increasing convex function on $[0,\infty)$, $\varPhi(0)=0$. Suppose there is $t_0>0$ such that $\varPhi(st)\le C\varPhi(s)\varPhi(t)$ for all $s,t\ge t_0$.
Then, the sequence of dilations and translations of every function $f\in L_{\varPhi}$, $\int_0^1f(t)\,dt\neq0$,
is an absolutely representing system in the Orlicz space $L_{\varPhi}$.
\end{corollary}

For Lorentz spaces it can be proved a more precise result by using the next theorem containing a useful necessary condition for a sequence of dilations and translations $\{V^{\alpha}f\}_{\alpha\in\mathbb{A}}$ to be an absolutely representing system (in contrast to Theorem \ref{t2}, without the additional requirement that each its tail part $\{V^{\alpha}f\}_{|\alpha|\ge k_0}$, $k_0=0,1,2,\dots$,  has this property as well).

\begin{theorem}\label{t3}
Let $X$ be a separable symmetric space and let the sequence of dilations and translations of a function $f\in X$, with $f=f^*$, be an absolutely representing system in $X$ with a constant $C$. Then, the following inequality holds
\begin{equation}\label{eq3.8}
\|\sigma_tf\|_X\le 2C|\langle f,1\rangle|\phi_X(t), \qquad 0<t\le1,
\end{equation}
where $\phi_X$ is the fundamental function of $X$.
\end{theorem}

\begin{proof}
First of all, in the same way as in the proof of the second part of Theorem \ref{t2}, we can prove inequality \eqref{eq3.6} but now only in the case when $k_0=0$, i.e.,
$$
\|y\|_{X^*}\le C\sup_{k\ge0}
\sup_{\|\{\xi_{\alpha}\}_{|\alpha|=k}\|_{\varXi_{k,f}}\le1}\biggl|\sum_{|\alpha|=k}\xi_{\alpha}\langle V^{\alpha}f,y\rangle\biggr|.
$$
Substituting here $y=V_1^n1=\chi_{[1-{2^{-n}},1]}$, we get
$$
\phi_{X^*}({2^{-n}})\le C\sup_{k\ge0}
\sup_{\|\{\xi_{\alpha}\}_{|\alpha|=k}\|_{\varXi_{k,f}}\le1}\biggl|\sum_{|\alpha|=k}\xi_{\alpha}\langle V^{\alpha}f,V_1^n1\rangle\biggr|.
$$
When calculating the right-hand side of this inequality, we consider two cases.

Firstly, let $k\ge n$.
Setting ${\bf1}_n:=\underbrace{(1,\dots,1)}_{n\; times}$, for $\alpha={\bf1}_n\gamma$, $|\gamma|=k-n$, we have
$\langle V^{\alpha}f,V_1^n1\rangle=2^{-k}|\langle f,1\rangle|$. Otherwise, $\langle V^{\alpha}f,V_1^n1\rangle=0$. Therefore, using \eqref{eq3.7} in the case when $\eta({\bf1}_n)=1$ and $\eta_{\beta}=0$ for $\beta\ne{\bf1}_n$, $|\beta|=n$, and also taking into account that $\|V_1^nf\|_X=\|\sigma_{{2^{-n}}}f\|_X$, we obtain
\begin{align*}
\sup_{\|\{\xi_{\alpha}\}_{|\alpha|=k}\|_{\varXi_{k,f}}\le1}\biggl|\sum_{|\alpha|=k}\xi_{\alpha}\langle V^{\alpha}f,V_1^n1\rangle\biggr|
=|\langle f,1\rangle|
\sup_{\|\{\xi_{\alpha}\}_{|\alpha|=k}\|_{\varXi_{k,f}}\le1}\biggl|\frac{1}{2^k}\sum_{|\gamma|=k-n}\xi({\bf1}_n\gamma)\biggr|\\
=\frac{|\langle f,1\rangle|}{2^n}\sup_{\|\xi({\bf1}_n)V_1^nf\|_X\le1}|\xi({\bf1}_n)|
=\frac{|\langle f,1\rangle|}{2^n\|\sigma_{{2^{-n}}}f\|_X}
\end{align*}

Let now $k<n$. Then, since $f=f^*$, for $\alpha={\bf1}_k$ we have
$$
\langle V^{\alpha}f,V_1^n1\rangle=\frac{1}{2^k}\int_{1-{2^{k-n}}}^1f(t)\,dt\le\frac{|\langle f,1\rangle|}{2^n},
$$
and $\langle V^{\alpha}f,V_1^n1\rangle=0$ for $\alpha\ne {\bf1}_k$, $|\alpha|=k$.
As a result, for all $k<n$
\begin{align*}
\sup_{\|\{\xi_{\alpha}\}_{|\alpha|=k}\|_{\varXi_{k,f}}\le1}\biggl|\sum_{|\alpha|=k}\xi_{\alpha}\langle V^{\alpha}f,V_1^n1\rangle\biggr|
\le\frac{|\langle f,1\rangle|}{2^n}\sup_{\|\xi({\bf{1}_k}) V_1^kf\|_X\le1}|\xi({\bf1}_k)|\\
=\frac{|\langle f,1\rangle|}{2^n\|\sigma_{{2^{-k}}}f\|_X}
\le\frac{|\langle f,1\rangle|}{2^n\|\sigma_{{2^{-n}}}f\|_X}.
\end{align*}
Putting all together, we see that
$$
\phi_{X^*}({2^{-n}})\le\frac{C|\langle f,1\rangle|}{2^n\|\sigma_{{2^{-n}}}f\|_X}.
$$
By the condition, the space $X$ is separable and so $X^*=X'$. Therefore, from a connection between the fundamental functions of a symmetric space and its K\"{o}the dual (see Section~\ref{Prel}a) it follows
$$
\phi_X({2^{-n}})=\frac{1}{2^n\phi_{X^*}({2^{-n}})}\ge\frac{\|\sigma_{{2^{-n}}}f\|_X}{C|\langle f,1\rangle|}.
$$
Since the functions $\|\sigma_tf\|_X$ and $\phi_X(t)$, $0\le t\le 1$, are quasi-concave for every symmetric space $X$  \cite[Theorems~II.4.5 and II.4.7]{KPS}, applying the standard reasoning, we come to inequality \eqref{eq3.8}.
\end{proof}

By Theorem \ref{t3}, we are able to give necessary and sufficient conditions for the sequence of dilations and translations of a decreasing function $f$ from a Lorentz space $\varLambda_{\varphi}$  to be an absolutely representing system in $\varLambda_{\varphi}$. For this we will need the following property of Lorentz spaces \cite[Lemma II.5.2]{KPS}.

\begin{proposition}\label{p3}
If a convex functional $J:\varLambda_{\varphi}\to[0,\infty]$ is bounded on the set of characteristic functions, i.e., for some $C>0$ and all $E\subset [0,1]$ the inequality $J(\chi_E)\le C\varphi(mE)$ holds, then $J$ is bounded on the whole space $\varLambda_{\varphi}$.
\end{proposition}

\begin{theorem}\label{c1}
Let $\varphi$ be an increasing concave function on $[0,1]$, $\varphi(0)=0$, and let $0\ne f\in\varLambda_{\varphi}$, $f^*=f$.
Then, the sequence of dilations and translations of $f$ is an absolutely representing system in the space $\varLambda_{\varphi}$ if and only if $f\in\mathscr{M}(\varLambda_{\varphi})$. 
\end{theorem}

\begin{proof}
In view of Theorem \ref{t2}, we need to prove only the necessity of the condition $f\in\mathscr{M}(\varLambda_{\varphi})$.

We consider the convex functional $J(x):=\|f\otimes x\|_{X(I\times I)}$, $x\in\varLambda_{\varphi}$. One can easily see that
$$
m_2\{s,t\in[0,1]:|f(s)|\chi_E(t)>\tau\}=m\{s\in[0,1]:|\sigma_{m(E)}f(s)|>\tau\},
$$
which means that the functions $f\otimes\chi_E$ and $\sigma_{m(E)}f$ are equimeasurable.
Consequently, from inequality \eqref{eq3.8} of Theorem \ref{t3} it follows
$$
J(\chi_E)=\|f\otimes\chi_E\|_{X(I\times I)}=\|\sigma_{m(E)}f\|_X\le C\varphi(m(E)).
$$
Finally, applying Proposition \ref{p3}, we conclude that the operator $B_fx=f\otimes x$ is bounded from $\varLambda_{\varphi}$ into $\varLambda_{\varphi}(I\times I)$, i.e., $f\in\mathscr{M}(\varLambda_{\varphi})$.
\end{proof}

Since a multiplicator space $\mathscr{M}(X)$ is symmetric and the tensor product is bounded from $\varLambda_{\varphi}\times \varLambda_{\varphi}$ into $\varLambda_{\varphi}(I\times I)$ if and only if the function $\varphi(t)$ is submultiplicative (see Section~\ref{Prel}b), we have

\begin{corollary}\label{c2}
Let $\varphi$ be an increasing convex function on $[0,1]$, $\varphi(0)=0$. The following conditions are equivalent:

(i) each function $f\in\varLambda_{\varphi}$,
$\int_0^1f(t)\,dt\neq0$, generates an absolutely representing system of dilations and translations in the Lorentz space $\varLambda_{\varphi}$;

(ii) the function $\varphi(t)$ is submultiplicative.
\end{corollary}

\section{A property of frames with respect to $\ell^1$-sums of finite-dimensional spaces.}

\begin{definition}\label{d1.1}
A frame $\{x_n\}_{n=1}^{\infty}$ in a Banach space $X$ with respect to a sequence space $\varDelta$ is said to be projective if there exist a Banach space $Y$ and a basis $\{e_n\}_{n=1}^{\infty}$ in the direct sum $X\times Y$, which is equivalent to the unit vector basis $\{\delta_n\}_{n=1}^{\infty}$ in $\varDelta$, such that $x_n=Pe_n$ for all $n=1,2,\dots$, where $P:X\times Y\to X$ is the canonical projection of $X\times Y$ onto $X$.
\end{definition}

Recall that each Duffin-Schaeffer frame is projective \cite{CHL99}. In the case of Banach spaces we can state the following criterion, which is a consequence of some general geometric principles (cf. \cite{Ter10}).

\begin{proposition}\label{p2}
Suppose $\{x_n\}_{n=1}^{\infty}$ is a frame in a Banach space $X$ with respect to a space $\varDelta$. The following conditions are equivalent:

(i) $\{x_n\}_{n=1}^{\infty}$ is a  projective frame;

(ii) there is a sequence $\{y_n\}_{n=1}^{\infty}\subset X^*$ such that for each $x\in X$ we have that
$\{\langle x,y_n\rangle\}_{n=1}^{\infty}\in\varDelta$ and $x=\sum_{n=1}^{\infty}\langle x,y_n\rangle x_n$;

(iii) the subspace $N(\varDelta)=\{\xi\in\varDelta:\sum_{n=1}^{\infty}\xi_nx_n=0\}$ is complemented in $\varDelta$.
\end{proposition}

In the proofs of Section~\ref{Frame} we made use of frames with respect to $\ell^1$-sums of finite-dimensional spaces.  It turns out that every such a frame in a symmetric space is not projective.

\begin{theorem}\label{l2}
Every frame in a symmetric space $X$ with respect to a $\ell^1$-sum of finite-dimensional spaces fails to be projective.
\end{theorem}

\begin{proof}
On the contrary, assume that there exists a projective frame in a symmetric space $X$ with respect to a $\ell^1$-sum $\varDelta$ of finite-dimensional spaces. Then, according to Proposition \ref{p2}, the subspace $N=N(\varDelta)$ is complemented in $\varDelta$ and hence $\varDelta=N\oplus M$.
It is clear that the restriction of the surjective analysis operator $S:\varDelta\to X$ to the complementary subspace $M$ 
is an isomorphism from $M$ onto $X$, whence
$M$ is isomorphic to $X$. On the other hand, in \cite[p.~19]{Bou} Bourgain proved that an arbitrary $\ell_1$-sum of finite-dimensional Banach spaces possesses the Schur property (recall that a Banach space $Y$ has the Schur property if  weak convergence of a sequence in $Y$ implies its $Y$-norm convergence).
Therefore, the space $\varDelta$ as well its subspace $M$ has this property. At the same time, it is known (see \cite{KM00}) that every symmetric function space fails to have the Schur property. Thus, since the latter property is preserved under isomorphisms, we obtain a contradiction with the fact that $M$ is isomorphic to $X$.
\end{proof}

\begin{corollary}\label{c2}
There is no symmetric space $X$ such that for some function $f\in X$ and all $x\in X$ we have
\begin{equation}\label{eq4.1}
x=\sum_{\alpha\in\mathbb{A}}\langle x,g_{\alpha}\rangle V^{\alpha}f,\;\;\mbox{and}\;\;
\sum_{k=0}^{\infty}\biggl\|\sum_{|\alpha|=k}\langle x,g_{\alpha}\rangle V^{\alpha}f\biggr\|_X<\infty,
\end{equation}
with a fixed sequence $\{g_{\alpha}\}_{\alpha\in\mathbb{A}}\subset X^*$.
\end{corollary}

\begin{proof}
Assuming the contrary, suppose that for a symmetric space $X$ there are a function $f\in X$ and a sequence $\{g_{\alpha}\}_{\alpha\in\mathbb{A}}\subset X^*$ such that for each $x\in X$ a representation of the form \eqref{eq4.1} exists. Then, by estimate \eqref{eq2.3} and Proposition \ref{p2}, the system of dilations and translations of $f$ is a projective frame in $X$ with respect to
$\varXi=(\oplus_{k=0}^{\infty}\varXi_k)_{\ell^1}$, where $\varXi_k$ are coordinate copies of the subspaces of dyadic step functions of rank $k$. Since this contradicts Proposition \ref{l2}, desired result follows.
\end{proof}

\section{Appendix}

Here, we show that condition \eqref{main2}, playing a central role in the proof of the Filippov--Oswald theorem \cite{FO95}, is not satisfied by Lorentz spaces on $[0,1]$ different from $L_1$. This is an immediate consequence of the following connection of \eqref{main2} with the smoothness of a separable symmetric space on $[0,1]$ at the function, identically equal to $1$. Recall that a Banach space $E$ is {\it smooth}  at an element $x_0\in E$, $\|x_0\|_E=1$, whenever there exists a unique $x^*\in E^*$ with $\|x^*\|_{E^*}=x^*(x_0)=1$.

\begin{proposition}\label{final prop}
Let $X$ be a separable symmetric space on $[0,1]$. Then, condition \eqref{main2} is fulfilled for each $f\in X$ such that $\int_0^1 f(t)\,dt\neq 0$ if and only if $X$ is smooth  at $1$.
\end{proposition}
\begin{proof}
Firstly, let condition \eqref{main2} be fulfilled for each $f\in X$ such that $\int_0^1 f(t)\,dt\neq 0$. Assuming that $X$ is not smooth at $1$, we find two functions $y_1$ and $y_2$, $y_1\ne y_2$, from the dual space $X^*=X'$ such that
$$
\|y_1\|_{X^*}=\|y_2\|_{X^*}=\langle 1,y_1\rangle=\langle 1,y_2\rangle=1.$$
Let $f\in X$ be an arbitrary function such that $a:=\langle f,y_1\rangle>0$ and $b:=-\langle f,y_2\rangle>0$. Obviously, we can assume that $\int_0^1 f(t)\,dt\neq 0$. Then, we have
$$
\|1-\lambda f\|_{X}\ge\langle 1-\lambda f,y_1\rangle=1-\lambda a\ge 1$$
if $\lambda\le 0$, and similarly
$$
\|1-\lambda f\|_{X}\ge\langle 1-\lambda f,y_2\rangle=1+\lambda b\ge 1$$
if $\lambda>0$.
This contradicts the condition.

Conversely, suppose that $X$ is smooth at $1$ but, however, there is a function $f\in X$, $\int_0^1 f(t)\,dt\neq 0$, such that
$$
\|1-\lambda f\|_X\ge 1\;\;\mbox{for all}\;\lambda\in\mathbb{R}.$$
Then, clearly, the projection $P(a\cdot 1+b\cdot f):= a\cdot 1$, $a,b\in\mathbb{R}$, defined on the subspace, spanned by  $1$ and $f$, has norm $1$. Therefore, by Hahn-Banach Theorem, we have
$$
1=\inf_{\lambda\in\mathbb{R}}\|1-\lambda f\|_X=\inf_{\lambda\in\mathbb{R}}\sup_{\|y\|_{X^*}\le 1}|\langle 1-\lambda f,y\rangle|=
\sup_{\|y\|_{X^*}\le 1,\langle f,y\rangle=0}|\langle 1,y\rangle|.
$$
Hence, there exists a sequence $\{y_n\}\subset X^*=X'$ such that $\|y_n\|_{X^*}\le 1$, $\langle f,y_n\rangle=0$, $n=1,2,\dots$,  and $\langle 1,y_n\rangle\to 1$ as $n\to\infty$. Since the closed unit ball in $X^*$ is weakly$^*$ compact, we can find a subsequence $\{y_{n_k}\}\subset \{y_n\}$ such that $y_{n_k}\to \tilde{y}$ weakly$^*$
for some $\tilde{y}\in X^*$, $\|\tilde{y}\|_{X^*}\le 1$.
This implies that $\langle f,\tilde{y}\rangle=0$ and $\langle 1,\tilde{y}\rangle=1$.
On the other hand, since $\|x\|_1\le \|x\|_X$ (see
Section~\ref{Prel}a), we have
$$
\|1\|_{X^*}=\langle 1,1\rangle =1.$$
Therefore, taking into account that $X$ is smooth at $1$, from the preceding equations we deduce that $\tilde{y}(t)\equiv 1$ and so $\langle f,1\rangle=\int_0^1 f(t)\,dt=0$, which contradicts the hypothesis.

\end{proof}


\begin{corollary}\label{l2}
Let $\varphi$ be an increasing convex function on $[0,1]$, $\varphi(0)=0$, $\varphi(1)=1$, and $\lim_{t\to 0}\varphi(t)/t=\infty$.
Then there is a function $f\in\varLambda_{\varphi}$ such that $\int_0^1 f(t)\,dt\neq 0$ and for each $\lambda\in\mathbb{R}$ we have
$$
\|1-\lambda f\|_{\varLambda_{\varphi}}\ge 1.$$
\end{corollary}
\begin{proof}
Recall that isometrically $(\varLambda_{\varphi})^*=M_{\varphi}$, where $M_{\varphi}$ is the Marcinkiewicz space with the norm
$$
\|x\|_{M_{\varphi}}=\sup_{0<t\le1}\frac{1}{\varphi(t)}\int_0^tx^*(s)\,ds
$$
\cite[Theorem~II.5.2]{KPS}. One can easily check that from properties of $\varphi$ it follows that both functions $y_1(t)\equiv 1$ and $y_2(t)=\varphi'(t)$ belong to $M_{\varphi}$, $y_1\ne y_2$, and
$$
\|y_1\|_{M_{\varphi}}=\|y_2\|_{M_{\varphi}}=\langle 1,y_1\rangle=\langle 1,y_2\rangle=1.$$
This means that the space $\varLambda_{\varphi}$ is not smooth at $1$. Therefore, applying
Proposition~\ref{l2}, we get desired result.
\end{proof}

\begin{remark}\label{Rem2}
A careful inspection of the proof of Lemma~2 from the paper \cite{FO95} shows that, in fact, this proof is based on using the well-known Weak Greedy Algorithm. In the case of $L_p$, $1\le p<\infty$, everything that is needed to apply it is condition \eqref{main2}. However, if we try to prove an analogue of the Filippov--Oswald theorem for a general separable symmetric space $X$ on $[0,1]$, the following much more restrictive conditions are required:

(a) $f\in \mathscr{M}(X)$;

(b) $dist_{\mathscr{M}(X)}(1,X_{0,f})<1$;

(c) $\sup_{\|x\|_{\mathscr{M}(X)}\le 1}\liminf_{k\to\infty}dist_{\mathscr{M}(X)}(x,X_{k,f})<1$.

\noindent
Here, as before, $X_{k,f}=\text{span}((f_{\alpha})_{|\alpha|=k}])$, $k=0,1,2,\dots$, and for every Banach space $Y$, $L\subset Y$ and $y_0\in Y$ we set
$$
dist_{Y}(y_0,L):=\inf_{y\in L}\|y_0-y\|_Y.$$


In contrast to that, according to Theorem \ref{t2}, the only condition $f\in \mathscr{M}(X)$ (together with \eqref{main1}) assures that the sequence of dyadic dilations and translations of $f$ is an absolutely representing system in the separable symmetric space $X$. Thus, we see that the frame approach, used in this paper, works under less restrictive conditions and so has wider applicability than the above Weak Greedy Algorithm, used in \cite{FO95} (cf. \cite{Sil08}).
\end{remark}

{\bf Acknowledgements.}
{\it The work of the first author was supported by the Ministry of Education and Science of the Russian Federation, project 1.470.2016/1.4 and by the RFBR grant 18-01-00414.

The work of the second author was supported by the RFBR grant 18-01-00414.}

\end{document}